\documentclass[a4paper,12pt]{article}

\usepackage{amsmath, wrapfig}
\usepackage{dsfont, a4wide, amsthm, amssymb, amsfonts, graphicx}
\usepackage{fancyhdr, xspace, psfrag, setspace, supertabular, color, enumerate}
\usepackage{hyperref}

\newtheorem{theorem}{Theorem}[section]
\newtheorem{proposition}[theorem]{Proposition}

\newtheorem{corollary}[theorem]{Corollary}

\newtheorem{definition}[theorem]{Definition}

\newtheorem{example}[theorem]{Example}
\newtheorem{question}[theorem]{Question}

\def\a{\alpha}
\def\b{\beta}
\def\g{\gamma}

\def\e{\epsilon}

\def\cf{\mathcal{F}}
\def\ct{\mathcal{T}}

\def\F{\mathbb{F}}

\DeclareMathOperator{\tr}{tr}

\DeclareMathOperator{\pr}{pr}
\DeclareMathOperator{\Rrk}{\mathit{R}rk}

\DeclareMathOperator{\Rwedgerk}{(\mathit{R}_{1} \wedge \mathit{R}_{2})rk}

\DeclareMathOperator{\Ronerk}{\mathit{R}_{1}rk}
\DeclareMathOperator{\Rtwork}{\mathit{R}_{2}rk}

\DeclareMathOperator{\Rirk}{\mathit{R}_{i}rk}
\DeclareMathOperator{\Rhrk}{\mathit{R_{[h]}}rk}

\title{The interplay between bounded ranks of tensors arising from partitions}
\author{Thomas Karam\footnote{Mathematical Institute, University of Oxford. Email: \texttt{thomas.karam@maths.ox.ac.uk}.}}

\begin{document}
\maketitle

\begin{abstract}

Let $d \ge 2, h \ge 1$ be integers. Using a fragmentation technique, we characterise $(h+1)$-tuples $(R_1, \dots, R_h, R)$ of non-empty families of partitions of $\{1, \dots, d\}$ such that it suffices for an order-$d$ tensor to have bounded $R_i$-rank for each $i=1,\dots,h$ for it to have bounded $R$-rank. On the way, we prove power lower bounds on products of identity tensors that do not have rank $1$, providing a qualitative answer to a question of Naslund.

\end{abstract}

\tableofcontents

\section{Introduction}

Throughout this paper, the notation $\F$ will be used to denote a field, and all our statements will be uniform with respect to the choice of the field. If $d \ge 1$ is an integer, and $n_1, \dots, n_d$ denote positive integers then we say that a function from $[n_1] \times \dots \times [n_d]$ to $\F$ is an \emph{order-$d$} tensor. For a fixed value of $d$, our statements will also be uniform with respect to the integers $n_1, \dots, n_d$ unless they explicitly appear in the relevant statement, and we will not redefine these integers.

\subsection{Ranks of tensors arising from partitions}

An order-$2$ tensor is a matrix, and ranks on matrices have been extensively studied. For higher-order tensors, there is no one single canonical generalisation of the matrix rank, and the interesting notion of rank instead depends on the application that one has in mind. Nonetheless, the definitions of many of them are similar in that for every nonnegative integer $k$, rank-$k$ tensors are defined in the same way in terms of rank-$1$ tensors, and only the sets of rank-$1$ tensors differ. Before discussing these notions of rank further, let us use that to define them in a unified way. Whenever $x$ is an element of $[n_1] \times \dots \times [n_d]$ and $J$ is a subset of $[d]$, we write $x(J)$ for the restriction of $J$ to its coordinates in $J$, that is, for the element $y$ of $\prod_{j \in J} [n_j]$ defined by $y_j = x_j$ for every $j \in [d]$.

\begin{definition} \label{Rrk definition} Let $d \ge 1$ be an integer, and let $R$ be a non-empty family of partitions of $\lbrack d \rbrack$. We say that an order-$d$ tensor $T$ has \emph{$R$-rank at most $1$} if there exist a partition $P \in R$ and for each $J \in P$ a function $a_J: \prod_{j \in J} [n_j] \rightarrow \mathbb{F}$ such that \[ T(x_1,\dots ,x_d) = \prod_{J \in P} a_J(x(J)) \] is satisfied for every $(x_1,\dots,x_d) \in [n_1] \times \dots \times [n_d]$. We say that the \emph{$R$-rank} of $T$ is the smallest nonnegative integer $k$ such that there exist order-$d$ tensors $T_1,\dots ,T_k$ each with $R$-rank at most $1$ and satisfying \[T = T_1 + \dots  + T_k.\] \end{definition}

We emphasize that the partition $P$ is not required to be the same for the tensors $T_1, \dots, T_k$. For instance, if $d=3$ and \[R = \{ \{\{1\}, \{2,3\}\}, \{\{2\}, \{1,3\}\} \},\] then a tensor of the type \[T(x_1,x_2,x_3) = a(x_1) b(x_2,x_3) + c(x_2) d(x_1,x_3)\] always has $R$-rank at most $2$.

We shall denote by $\Rrk T$ the $R$-rank of $T$. There are some choices of the family of partitions $R$ for which the corresponding notion of $R$-rank has received particular attention. We now briefly recall some of these notions and refer the reader to \cite{Tao}, \cite{Naslund}, \cite{Lovett}, \cite{K. decompositions} for a more complete discussion of the history and applications of the middle two.

We note that if $d=1$, then the family $\{\{\{1\}\}\}$ is the only non-empty family $R$ of partitions of [d], and every order-$1$ tensor hence has $R$-rank either $0$ (if it is the zero tensor) or $1$ (otherwise).

\begin{itemize}

\item The \emph{tensor rank}, corresponding to \[R = \{\{\{1\},\{2\},\dots, \{d\}\}\}\] has been heavily used in computational complexity theory.

\item The \emph{slice rank}, corresponding to \[R = \{\{\{1\},\{1\}^c\}, \{\{2\},\{2\}^c\}, \dots , \{\{d\},\{d\}^c\}\}\] was defined by Tao \cite{Tao} in 2016, in a reformulation of the breakthrough of Ellenberg and Gijswijt (building on work of Croot, Lev, and Pach) on the cap-set problem.

\item The \emph{partition rank}, corresponding to \[R = \{\{J,J^c\}: J \subset [d], 1 \le |J| \le d-1\}\] was defined by Naslund \cite{Naslund} in order to prove polynomial upper bounds on the size of subsets of $\F_{p^r}^n$ not containing $k$-right corners.

\item The \emph{$j$-flattening rank} for some $j \in [d]$, corresponding to \[R = \{\{\{j\},\{j\}^c\}\}\] also arises naturally from time to time (and also appears implicitly in many arguments involving projections of a tensor to a matrix). As an example of a recent application, the max-flattening rank, obtained by taking the maximum of the flattening ranks over all $j \in [d]$, was recently used by Correia, Sudakov and Tomon \cite{Correia Sudakov Tomon} to generalise the Frankl-Wilson theorem on forbidden intersections.

\end{itemize}

We shall denote the tensor rank and the partition rank by $\tr$ and $\pr$ respectively, and write $R_{\tr}$ and $R_{\pr}$ for the respective corresponding families $R$.

Comparing notions of rank, both qualitatively and quantitatively, is a theme that already pervades the literature on tensors. One fruitful line of enquiry involves the comparison between the partition rank, the definition of which we just recalled, with (in the finite field case) the analytic rank, a notion of rank which does not arise as a special case of Definition \ref{Rrk definition} but instead involves the bias of the associated multilinear form, and with the geometric rank, another notion of rank outside of the scope of Definition \ref{Rrk definition} and which is defined even if the field is infinite. 

The qualitative comparison between the partition and analytic ranks was settled in one direction by Gowers and Wolf \cite{Gowers and Wolf}, and essentially in the other, more difficult direction by Green and Tao in their paper \cite{Green and Tao} on the bias and rank of polynomials. Quantitatively, Lovett later proved \cite{Lovett} the stronger result that the analytic rank is at most the partition rank. Attempts to obtain the correct quantitative bounds in the converse direction then led to successive improvements which come very close to concluding that the partition rank is at most a multiple of the analytic rank: Janzer \cite{Janzer} and Milicevic \cite{Milicevic} independently obtained power bounds, which were then improved to linear bounds in the large fields case by Moshkovitz and Cohen \cite{Moshkovitz and Cohen}, and to quasilinear bounds for all finite fields by Moshkovitz and Zhu \cite{Moshkovitz and Zhu}.

Just before these last two papers the notion of geometric rank was defined by Kopparty, Moshkovitz and Zuiddam \cite{Kopparty Moshkovitz Zuiddam}, originally to prove sharp upper bounds on the border subrank of the matrix multiplication tensors, and was then shown in the same paper to be between the subrank of Strassen \cite{Strassen} and the slice rank. The geometric rank was then re-used in an important way in the previously mentioned improvement \cite{Moshkovitz and Cohen}, after which yet another notion, the local rank, was defined in \cite{Moshkovitz and Zhu}: comparing it to the analytic rank and to the partition rank and combining these comparisons then led to the main achievement of that paper.

Returning to the notions of rank in Definition \ref{Rrk definition}, rather little appears to be known about how they relate to one another, even purely qualitatively, and comparing these notions is the focus of the present paper.

\subsection{Two questions of Naslund}

The four notions of rank that we previously listed are by no means the only notions of rank that are worthwhile to study. As Naslund himself observed, a positive answer to the following question would improve lower bounds for non-commutative circuits, in particular Nisan’s noncommutative formula-size lower bound (\cite{Shpilka Yehudayoff}, Theorem 3.6). 

\begin{question}[\cite{Naslund}, Problem 12] \label{Naslund question on order-4 tensors} Let $R$ be the family of partitions on $[4]$ defined by \[R = \{\{\{1,2\},\{3,4\}\}, \{\{1,3\},\{2,4\}\}\}.\] Does there exist $\e>0$ such that the order-$4$ tensor $T: [n]^4 \to \F$ defined by \[T(x_1,x_2,x_3,x_4) = 1_{x_1=x_4} 1_{x_2=x_3}\] has $R$-rank at least $n^{1+\e}$, at least when $n$ is sufficiently large ? \end{question}

The lower bound of $n$ follows from considering the matrix obtained by fixing both $x_1$ and $x_4$ to be both equal to (say) $1$. Indeed, if we have a decomposition of the type \[\sum_{i=1}^r a_i(x_1,x_2) b_i(x_3,x_4) + \sum_{i=1}^s a_i(x_1,x_3) b_i(x_2,x_4) = 1_{x_1=x_4} 1_{x_2=x_3}\] for some nonnegative integers $r,s$ then in particular we have \[\sum_{i=1}^r a_i(1,x_2) b_i(x_3,1) + \sum_{i=1}^s a_i(1,x_3) b_i(x_2,1) = 1_{x_2=x_3}.\] When viewed as functions $[n_2] \times [n_3] \to \F$, each of the summands of the left-hand side has rank at most $1$ but the right-hand side has rank $n$, so $r+s \ge n$.

As Naslund points out, it follows from a counting argument (in the finite field case) that there exist order-$4$ tensors $[n]^4 \to \F$ which have $R$-rank $\Omega(n^2)$ (and in fact this counting argument shows that this is the case of a proportion of tensors approaching $1$ as $n$ tends to infinity), but no lower bounds of the type $\Omega(n^{1+\e}$) are known for the $R$-rank of any specific tensor.

This situation is in several ways far from an isolated case regarding lower bounds on the ranks of tensors. For instance, can be shown (again by a simple counting argument) that there exist order-$3$ tensors $[n]^3 \to \F$ with tensor rank $\Omega(n^2)$, but providing an example of a tensor for which this is the case is a famous open problem. There again the gap between that lower bound and the best known lower bound for a specific tensor is large. 

Whenever $d \ge 1$ is an integer and $P_1,P_2$ are two partitions of $[d]$, we say that $P_2$ is \emph{finer than} $P_1$ if for every $J_2 \in P_2$ there exists $J_1 \in P_1$ such that $J_2 \subset J_1$. We write this property as $P_2 \prec P_1$. We can check that $\prec$ is a partial order on partitions of $[d]$.

For every integer $d \ge 1$ and every non-empty subset $J$ of $[d]$ we write $\Delta_J: [n]^J \to \F$ for the tensor defined by $\Delta_J(x(J)) = 1$ if $x_{j_1} = x_{j_2}$ whenever $j_1,j_2 \in J$, and by $\Delta_J(x(J)) = 0$ otherwise (in the case where $J$ has size $1$, the tensor $\Delta_J$ is hence in particular identically equal to $1$). For every partition $P$ of $[d]$, we write $\Delta_P:[n]^d \to \F$ for the tensor defined by \[\Delta_P(x_1, \dots, x_d) = \prod_{J \in P} \Delta_J(x(J)).\] 

After asking Question \ref{Naslund question on order-4 tensors}, Naslund also posed the following yet bolder problem. 

\begin{question}[\cite{Naslund}, Problem 13] \label{Naslund question in the general case}

Let $d \ge 2$ be an integer, let $R$ be a non-empty family of partitions of $[d]$, and let $P$ be a partition of $[d]$ that is not finer than any of the partitions of $R$. What is the rank of $\Delta_P$ ?

\end{question}

If instead we assume that $P$ \emph{is} finer than at least one partition $P_0$ of $R$, then there is a short answer: there exists a function $s:P \to P_0$ such that $J \subset s(J)$ for every $J \in P$, and hence we can in particular write \[ \Delta_P(x_1, \dots, x_d) = \prod_{J' \in P_0} (\prod_{J \in s^{-1}(J')} \Delta_J)(x(J')), \] which shows that $T$ has $R$-rank equal to $1$.

In the special case where $P$ is the trivial partition $\{\{[d]\}\}$ the tensor $\Delta_P$ is the usual “identity" tensor $\Delta_d$ defined by \[\Delta_d(x_1, \dots, x_d) = \sum_{i=1}^n 1_{x_1=i} \dots 1_{x_d=i}.\] Naslund showed (\cite{Naslund}, Lemma 11) that the partition rank of $\Delta_d$ is equal to $n$. Because every non-trivial partition of $[d]$ is finer than some bipartition of $[d]$, it follows that whenever $R$ is a non-empty family of partitions of $[d]$ which does not contain the trivial partition $\{[d]\}$ we have $\Rrk T \ge \pr T$ for any order-$d$ tensor $T$, and hence in particular that $\Rrk \Delta_d \ge n$. (If instead $R$ contains the trivial partition, then $\Delta_d$ has $R$-rank equal to $1$, as does any non-zero order-$d$ tensor).

Nonetheless, for reasons which we will explain in more detail in Section \ref{Section: Lower bounds on ranks of products of diagonal tensors}, there does not appear to be a simple way of showing that $\Rrk \Delta_P \ge n$ or even $\Rrk \Delta_P = \Omega(n)$ in the general setting of Question \ref{Naslund question in the general case}, although such a linear bound is plausible to us. 

\subsection{Main results}

Our first result will be a lower bound of power type for Question \ref{Naslund question in the general case}.

\begin{theorem} \label{Diagonal bound theorem} Let $d \ge 2$ be an integer, let $P$ be a partition of $[d]$, and let $R$ be a non-empty family of partitions of $[d]$. If $P$ is not finer than any of the partitions in $R$, then the order-$d$ tensor $\Delta_P$ satisfies $\Rrk \Delta_P \ge n^{c(d)}$ with $c(d) = 3^{-2^{d+1}}$. \end{theorem}

The partial order that we defined on partitions can be used to define a relation on non-empty families of partitions. Whenever $d \ge 1$ is an integer and $R_1, R_2$ are two non-empty families of partitions of $[d]$, we say that $R_2$ is \emph{finer than} $R_1$ if for every partition $P_2 \in R_2$ there exists a partition $P_1 \in R_1$ satisfying $P_2 \prec P_1$. We will again write this property by $R_2 \prec R_1$. 

We note that if $R_2$ is contained in $R_1$, then $R_2 \prec R_1$, so whenever $R_2 \prec R_1$ we necessarily have both $R_1 \prec R_1 \cup R_2$ and $R_1 \cup R_2 \prec R_1$, so the relation $\prec$ is not antisymmetric, and hence not a partial order on families of partitions. It nonetheless satisfies reflexivity and transitivity, which is why we choose to keep the notation $\prec$. We also note that for $\prec$, the family $R_{\tr}$ is smaller than every other $R$ for $\prec$, the family $\{\{[d]\}\}$ is greater than every other $R$ for $\prec$, and the family $R_{\pr}$ is greater than every other $R$ not containing the trivial partition $[d]$.

With this definition, yet another observation of Naslund can be formulated as follows.

\begin{proposition}[{\cite{Naslund}, Proposition 9}] \label{Refinements do not decrease rank} Let $R_1, R_2$ be two non-empty families of partitions of $[d]$. If $R_2 \prec R_1$, then $\Ronerk T \le \Rtwork T$ for every order-$d$ tensor $T$. \end{proposition}

To prove that, it suffices to check that every order-$d$ tensor with $R_2$-rank at most $1$ must have $R_1$-rank at most $1$. Theorem \ref{Diagonal bound theorem} has the following consequence which may be viewed qualitatively as a converse to Proposition \ref{Refinements do not decrease rank}.

\begin{corollary} \label{Diagonal bound corollary} Let $d \ge 2$ be an integer, and let $R_1, R_2$ be two non-empty families of partitions of $[d]$. If $R_2 \nprec R_1$, then there exists an order-$d$ tensor $T$ satisfying $\Rtwork T = 1$ but $\Ronerk T \ge n^{c(d)}$. \end{corollary}

Besides the power bound in Theorem \ref{Diagonal bound theorem}, the present paper will be more broadly aimed at is the study of the relationships between the various notions of $R$-ranks of tensors. This paper may be viewed as being focused on the following question, which we will answer completely.

\begin{question} \label{Main question} Let $d \ge 2$, $h \ge 1$ be integers, and $R_1, \dots, R_h$ are non-empty families of partitions of $[d]$. For which non-empty sets $R$ of partitions of $[d]$ is it true that if $T$ is an order-$d$ tensor with bounded $R_i$-rank for each $i=1, \dots, h$, then $T$ must also have bounded $R$-rank ? \end{question}

We now give a definition in terms of which we will ultimately formulate our answer to Question \ref{Main question}.

\begin{definition} Let $d \ge 2$, $h \ge 1$ be integers. If $P_1, \dots, P_h$ are partitions of $[d]$ (or more generally of some subset of $[d]$), then we say that the \emph{least common refinement} of $P_1, \dots, P_h$ is the partition \[P_1 \wedge \dots \wedge P_h = \{J_1 \cap \dots \cap J_h: J_1 \in P_1, \dots, J_h \in P_h, J_1 \cap \dots \cap J_h \neq \emptyset\}.\] If $R_1, \dots, R_h$ are non-empty families of partitions of $[d]$ (or more generally of some subset of $[d]$), then we say that the \emph{least common refinement} of $R_1, \dots, R_h$ is the family of partitions \[R_1 \wedge \dots \wedge R_h = \{P_1 \wedge \dots \wedge P_h: P_1 \in R_1, \dots, P_h \in R_h \}.\] \end{definition}

We note the identities \begin{align*} P_h \wedge (P_1 \wedge \dots \wedge P_{h-1}) & = P_1 \wedge \dots \wedge P_h \\
R_h \wedge (R_1 \wedge \dots \wedge R_{h-1}) & = R_1 \wedge \dots \wedge R_h \end{align*} which will very shortly be useful to us in reducing situations involving several families of partitions to situations involving only two such families.

Theorem \ref{Diagonal bound theorem} together with the following theorem will suffice for us to answer Question \ref{Main question} in full.

\begin{theorem}\label{Boundedness for two ranks} Let $d \ge 2$, $k_1, k_2 \ge 1$ be integers, and let $R_1,R_2$ be non-empty families of partitions of $[d]$. There exists an integer $C_{d,R_1.R_2}(k_1, k_2)$ such that if $T$ is an order-$d$ tensor such that $\Ronerk T \le k_1$ and $\Rtwork T \le  k_2$, then \[\Rwedgerk T \le C_{d,R_1.R_2}(k_1, k_2).\] \end{theorem}

Corollary \ref{Diagonal bound corollary} and Theorem \ref{Boundedness for two ranks} together establish the following dichotomy and answer to Question \ref{Main question}.

\begin{theorem}\label{Main theorem} Let $d \ge 2$, $h \ge 1$ be integers, let $R_1, \dots, R_h$ be non-empty families of partitions of $[d]$, and let $R_{[h]}$ be the family of partitions $R_1 \wedge \dots \wedge R_h$. For every non-empty set $R$ of partitions of $[d]$ the following holds.

\begin{enumerate}

\item If $R_{[h]} \nprec R$, then there exists an order-$d$ tensor $T:[n]^d \to \F$ satisfying $\Rirk T = 1$ for each $i \in [h]$ but $\Rrk T \ge n^{c(d)}$.

\item If $R_{[h]} \prec R$, then whenever $k_1, \dots, k_h \ge 1$ are integers there exists some integer $C_{d,(R_1,\dots, R_h)}(k_1,\dots,k_h)$ such that if $T$ is an order-$d$ tensor satisfying $\Rirk T \le k_i$ for each $i \in [h]$, then \[\Rrk T \le C_{d,(R_1,\dots, R_h)}(k_1,\dots,k_h).\]

\end{enumerate}

\end{theorem}

\begin{proof}

If $R_{[h]} \nprec R$ then there exists some partition $P \in R_{[h]}$ which is not finer than any partition of $R$. Taking $T = \Delta_P$ and applying Theorem \ref{Diagonal bound theorem} to $R$ and $P$ then provides the conclusion of Item 1. 

If $R_{[h]} \prec R$ then by Proposition \ref{Refinements do not decrease rank}, for every order-$d$ tensor $T$ we have $\Rrk T \le \Rhrk T$, so it suffices to show Item 2 in the case $R = R_{[h]}$. This in turn is established by Theorem \ref{Boundedness for two ranks} and an induction on $h$. \qedhere

\end{proof}

Theorem \ref{Main theorem} in particular provides a sufficient condition for a tensor to have bounded tensor rank.

\begin{corollary} Let $d \ge 2, h \ge 1$ be integers, let $R_1,\dots,R_h$ be non-empty families of partitions of $[d]$ such that $P_1 \wedge \dots \wedge P_h$ is the discrete partition for all $P_1 \in R_1, \dots, P_h \in R_h$. Then, for any integers $k_1, \dots, k_h \ge 1$, if $T$ is an order-$d$ tensor with $\Rirk T \le k_i$ for each $i \in [h]$, then \[\tr T \le C_{d,(R_1,\dots, R_h)}(k_1,\dots,k_h).\] \end{corollary}

Section \ref{Section: The fragmentation technique} will be devoted to a fragmentation technique which will then be our central tool in the proofs of both Theorem \ref{Diagonal bound theorem} and Theorem \ref{Boundedness for two ranks} in Section \ref{Section: Lower bounds on ranks of products of diagonal tensors} and Section \ref{Boundedness of ranks generated by a pair of ranks} respectively.

\section{The fragmentation technique} \label{Section: The fragmentation technique}

This section discusses a fragmentation result which will allow us to fragment notions of rank. Before doing that, let us introduce some notation on tensors which we will use throughout the present paper.

Let $d \ge 2$ be an integer. If $T$ is an order-$d$ tensor, $J$ is a subset of $[d]$ and $y$ is an element of $\prod_{j \in J^c} [n_j]$, then we write $T_y: \prod_{j \in J} [n_j] \to \F$ for the order-$|J|$ tensor defined by \[T_y(x(J)) = T(z)\] where $z$ is the element of $\prod_{j=1}^d [n_j]$ defined by $z_j = x_j$ for every $j \in J$ and $z_j = y_j$ for every $j \in J^c$.

If $\{J_1, J_2\}$ is a bipartition of $[d]$, $T$ is an order-$d$ tensor and $y \in \prod_{j \in J_1} [n_j]$, $z \in \prod_{j \in J_2} [n_j]$ are elements, then we write $T(y,z)$ for the value $T(x)$ where $x$ is the element of $\prod_{j=1}^d [n_j]$ defined by $x_j = y_j$ for every $j \in J_1$ and $x_j = z_j$ for every $j \in J_2$.

If $J_1,J_2$ are subsets of $\lbrack d \rbrack$, $T_1: \prod_{j \in J_1} [n_j] \rightarrow \mathbb{F}$, $T_2: \prod_{j \in J_2} [n_j] \rightarrow \mathbb{F}$ are tensors, then we write $T_1.T_2: \prod_{j \in J_1 \Delta J_2} [n_j] \rightarrow \mathbb{F}$ for the tensor (or rather element of $\F$ if $J_1 = J_2$) defined by \[ (T_1.T_2)(y,z) = \sum_{x \in \prod_{j \in J_1 \cap J_2} [n_j]} T_1(y,x) T_2(z,x)  \] for each $y \in \prod_{j \in J_1 \setminus J_2} [n_j]$ and each $z \in \prod_{j \in J_2 \setminus J_1} [n_j]$.

If $r$ is a nonnegative integer and $A_1, \dots, A_r, A$ are linearly independent functions from some finite set $X$ (of the type $\prod_{j \in J} [n_j]$) to $\F$, then it follows from Gaussian elimination that there exists a function $u:X \to \F$ such that $u.A_i = 0$ for every $i \in [r]$ but $u.A = 1$, and that we can require the function $u$ to have support size at most $r+1$, that is, that $u^{-1}(\F \setminus \{0\})$ contains at most $r+1$ elements of $X$. Likewise, there exist functions $A_1^*, \dots, A_r^*: X \to \F$ with supports all contained in the same set of size $r$ such that $A_i^*.A_j = 1_{i=j}$ for all $i,j \in [r]$. We will refer to the family $(A_1^*, \dots, A_r^*)$ as a family of \emph{dual functions} to the family $(A_1, \dots, A_r)$.

If $R$ is a non-empty family of partitions of $[d]$, then we say that an \emph{R-rank decomposition} of an order-$d$ tensor $T$ is a decomposition of the type \begin{equation} T(x) = \sum_{P \in R} \sum_{i=1}^{r_P} \prod_{J \in P} a_{P,i,J}(x(J)). \label{R-rank decomposition} \end{equation}

We will refer to the integer $\sum_{P \in R} r_P$ as the \emph{length} of the decomposition. We will say that an \emph{R-rank decomposition} of $T$ has minimal length if its length has the smallest possible value, that is, the $R$-rank of $T$.

Our proofs will often single out a particular subset $J$ of $[d]$ as playing a particular role. To that purpose, given $R$ a non-empty family of partitions of $[d]$ and $J$ a subset of $[d]$ satisfying $1 \le |J| \le d-1$ we define the auxiliary families of partitions \begin{align*} R_+(J) & = \{Q \in R: J \in Q\} \\
R_-(J)  & = \{Q \in R: J \notin Q\} \\
R_{\mathrm{comp}}(J)  & = \{Q \setminus \{J\}: Q \in R_+(J)\} \\
R_{\mathrm{new}}(J)  & = \{Q \cup \{J_1,J_2\}\}: Q \in R_{\mathrm{comp}}(J), \{J_1, J_2\} \text{ a bipartition of } J\} \\
R'(J) & = R_-(J)  \cup R_{\mathrm{new}}(J). \end{align*}

Informally, the family of partitions $R’(J)$ consists of $R$ modified in that the part $J$ is further split in all possible ways into two non-empty parts whenever it appeared in a partition of $R$. (The index “comp" in $R_{\mathrm{comp}}(J)$ stands for “compatible".) This way of viewing the partitions of $R$ had already been applied in \cite{K. subtensors}, Section 11 to obtain a high-rank subtensor of bounded size from a high-rank tensor.

It will be convenient for us to use a total ordering $\le$ on non-empty subsets of $[d]$ by choosing it arbitrarily on the class of sets \[\{J \subset [d]: |J|=t\}\] for every $1 \le t \le d$ and then specifying that whenever $J_1,J_2$ are subsets of $[d]$ with $1 \le |J_1| < |J_2|$ we have $J_1 < J_2$. For every non-empty family $R$ of partitions of $[d]$ we let \[J_{\max}(R) = \max \{J \in P: P \in R\}\] where the maximum is taken with respect to the total ordering that we have just defined. 

We will use this ordering to prove properties inductively on families of partitions $R$ where $J_{\max}(R)$ is at most some subset $J$ of $[d]$. At the inductive step the fragmentation technique, encapsulated by Proposition \ref{Fragmentation}, will allow us to reduce a statement which we want to prove for some $R$ to a statement for the family $R'=R'(J_{\max}(R))$, which satisfies $J_{\max}(R') < J_{\max}(R)$. As the number of possibilities for $J_{\max}(R)$ is at most $2^d$, this step will have to be performed at most that many times, after which the resulting family of partitions that we have obtained may only contain the discrete partition $R_{\tr}$, in which case it is then not difficult to conclude.

If $J$ is a non-empty strict subset of $[d]$, then the decomposition \eqref{R-rank decomposition} can be written more concisely as \begin{equation} T(x) = \sum_{i=1}^r A_i(x(J^c)) B_i(x(J)) + \sum_{i=1}^{s} F_i(x), \label{R-rank decomposition with A_i, B_i, F_i} \end{equation} where $r,s$ are nonnegative integers, the functions $A_i: \prod_{j \in J^c} [n_j] \to \F$ have $R_{\mathrm{comp}}(J)$-rank at most $1$, the functions $B_i: \prod_{j \in J} [n_j] \to \F$ are arbitrary, and the functions $F_i:\prod_{j=1}^d [n_j] \to \F$ have $R_-(J)$-rank at most $1$. We may furthermore require the functions $A_1, \dots, A_r$ to be linearly independent, and when the decomposition \eqref{R-rank decomposition with A_i, B_i, F_i} has minimal length this will always be the case. Indeed if for instance, $A_r$ were linearly spanned by $A_1, \dots, A_{r-1}$, then we would be able to rewrite the first sum as \[\sum_{i=1}^{r-1} A_i(x(J^c)) B_i'(x(J))\] for some new functions $B_1', \dots, B_{r-1}'$.

\subsection{A few examples}

The basic goal of the technique which we will be describing in the remainder of this section is a way to use two decompositions of a tensor to deduce decompositions that are more fragmented than the original two. Let us begin with a few examples to illustrate the underlying idea.

\begin{example} \label{Example 1 fragmentation}

\emph{Assume that \begin{equation} \sum_{i=1}^r a_i(x) b_i(y,z) = \sum_{j=1}^s c_j(y) d_j(x,z) \label{equality between two order-3 decompositions} \end{equation} are two decompositions of the same order-$3$ tensor $T$, where the family $(a_1, \dots, a_r)$ is linearly independent. Let $(a_1^*, \dots, a_r^*)$ be a family of dual functions to the family $(a_1, \dots, a_r)$. For every $i \in [r]$, applying the function $a_i^*$ to \eqref{equality between two order-3 decompositions} we obtain \[b_i(y,z) = \sum_{j=1}^s (a_i^*.(c_j d_j))(y,z).\] We may factor \[(a_i^*.(c_j d_j))(y,z) = c_j(y) (a_i^*.d_j)(z)\] and hence write \[T(x,y,z) = \sum_{i=1}^r \sum_{j=1}^s a_i(x) c_j(y) (a_i^*.d_j)(z)\] which shows in particular that $T$ has tensor rank at most $rs$.}

\end{example}

\begin{example} \label{Example 2 fragmentation}

\emph{Let $d \ge 3$ be an integer and let $T$ be an order-$d$ tensor which admits a decomposition \begin{equation} T(x_1, \dots, x_d) = \sum_{i=1}^{k_j} a_{j,i}(x_j) b_{j,i}(x_1, \dots, x_{j-1}, x_{j+1}, \dots, x_d) \label{flattening decompositions} \end{equation} for every $j \in [d]$. In other words, we are assuming that the $j$-flattening rank of $T$ is at most $k_j$ for every $j \in [d]$. Also assume that for every $j \in [d]$ the family $(a_{j,1}, \dots, a_{j,k_j})$ is linearly independent and let $(a_{j,1}^*, \dots, a_{j,k_j}^*)$ be a family of dual functions to this family. Applying the functions $a_{1,i}^*$ to the equality between the decompositions \eqref{flattening decompositions} with $j=1$ and $j=2$ we obtain \[T(x_1, \dots, x_d) = \sum_{i_1=1}^{k_1} \sum_{i_2=1}^{k_2} a_{1,i_1}(x_1) a_{2,i_2}(x_2) b_{(1,i_1), (2,i_2)}(x_3, \dots, x_d),\] where \[b_{(1,i_1), (2,i_2)} = a_{1,i_1}^*.b_{2,i_2} + a_{2,i_2}^*.b_{1,i_1}.\] Writing the equality between this decomposition and the decomposition \eqref{flattening decompositions} with $j=3$, and applying the functions $a_{1,i_1}^* \otimes a_{2,i_2}^*$ we obtain a decomposition of the type \[T(x_1, \dots, x_d) = \sum_{i_1=1}^{k_1} \sum_{i_2=1}^{k_2} \sum_{i_3=1}^{k_3} a_{1,i_1}(x_1) a_{2,i_2}(x_2) a_{3,i_3}(x_3) b_{(1,i_1), (2,i_2), (3,i_3)}(x_4, \dots, x_d).\] Iterating further (the next step starting with applying the functions $a_{1,i_1}^* \otimes a_{2,i_2}^* \otimes a_{3,i_3}^*$) we ultimately obtain \[T(x_1, \dots, x_d) = \sum_{i_1 \in [k_1], \dots, i_d \in [k_d]} \lambda_{i_1, \dots, i_d} a_{1,i_1}(x_1) a_{2,i_2}(x_2) \dots a_{d,i_d}(x_d),\] and we recover the known fact that $T$ has tensor rank at most $k_1 \dots k_{d-1}$. (To establish this last conclusion it suffices to run this argument except the last iteration. In turn, to be able to do that it suffices to assume that the $j$-flattening ranks are bounded above by $k_j$ for all $j \in [d-1]$ rather than for all $j \in [d]$.)}

\end{example}

\begin{example} \label{Example 3 fragmentation}

\emph{Assume that \[ \sum_{i=1}^r a_i(x,y) b_i(z,w) = \sum_{j=1}^s c_j(y,z) d_j(x,w) \] are two decompositions of the same order-$4$ tensor $T$, where the family $(a_1, \dots, a_r)$ is linearly independent. We then let $(a_1^*, \dots, a_r^*)$ be a family of dual functions to this family, and furthermore require that the functions $a_1^*, \dots, a_r^*$ are all supported inside some subset $U \subset [n_1] \times [n_2]$ with size $r$. Then as in Example \ref{Example 1 fragmentation}, for every $i \in [r]$ we can write \[b_i(z,w) = \sum_{j=1}^s (a_i^*.(c_jd_j))(z,w).\] However, the expression $(a_i^*.(c_jd_j))(z,w)$ does not immediately factor as it did in Example \ref{Example 1 fragmentation}: indeed we have \begin{equation} (a_i^*.(c_jd_j))(z,w) = \sum_{(x,y) \in [n_1] \times [n_2]} a_i^*(x,y) c_j(y,z) d_j(x,w), \label{expression that does not factor} \end{equation} and the set $\{1,2\}$ of coordinates on which $a_i^*$ depends is not contained in any the respective corresponding sets $\{2,3\}$ and $\{1,4\}$ of $c_j$ and $d_j$. Instead we use in an important way the fact that $a_i^*$ is supported inside $U$. Indeed the right-hand side of \eqref{expression that does not factor} has a sum that can be taken merely over $U$, and is hence equal to the linear combination \[ \sum_{(u_1, u_2) \in U} a_i^*(u_1, u_2) c_j(u_2,z) d_j(u_1,w)\] of $r$ functions with rank at most $1$. We can then write \[b_i(z,w) = \sum_{j=1}^s \sum_{(u_1,u_2) \in U} a_i^*(u_1,u_2) c_j(u_2,z) d_j(u_1,w).\] Likewise we can choose a dual family $(b_1^*, \dots, b_r^*)$ to $(b_1, \dots, b_r)$ supported inside some subset $V$ of $[n_3] \times [n_4]$, and then write \[a_i(x,y) = \sum_{j=1}^s \sum_{(u_3,u_4) \in V} b_i^*(u_3,u_4) c_j(y,u_3) d_j(x,u_4).\] We hence obtain \begin{multline} T(x,y,z,w) = \sum_{(u_1,u_2) \in U} \sum_{(u_3,u_4) \in V} \left(\sum_{i=1}^r a_i^*(u_1,u_2) b_i^*(u_3,u_4) \right) \\ \left( \sum_{j'=1}^s \sum_{j''=1}^s d_{j'}(x,u_4) c_{j'}(y,u_3) c_{j''}(u_2,z) d_{j''}(u_1,w) \right), \nonumber \end{multline} which shows in particular that the tensor rank of $T$ is at most $|U||V|s^2 \le r^2s^2$.} \end{example}

\begin{example} \label{Example 4 fragmentation}

\emph{Assume that \[\sum_{i=1}^r a_i(x,y) b_i(z,w) = \sum_{j=1}^s c_j(x) d_j(y,z,w)\] are two decompositions of the same order-$4$ tensor $T$. If we assume that the family $(a_1, \dots, a_r)$ is independent and consider a dual family $(a_1^*, \dots, a_r^*)$, applying some $a_i^*$ to both sides of the equality yields  \[b_i(z,w) = \sum_{j=1}^s \sum_{(x,y) \in [n_1] \times [n_2]} c_j(x) d_j(y,z,w),\] which does not allow us to write $b_i$ as a sum of a bounded number of rank-$1$ functions in general, even assuming that $a_i^*$ has support size at most $r$, because the set of variables on which the functions $d_j$ depend is not contained in the set of variables on which the functions $b_i$ depend. There exist tensors for which this is not the case, as can be seen for instance by taking \[T(x,y,z,w) = b(z,w)\] with $b$ of rank $n$.} \end{example}

Examples such as Example \ref{Example 4 fragmentation} are why whenever we will use decompositions such as \eqref{R-rank decomposition with A_i, B_i, F_i} with a view of fragmenting the functions therein, we will always choose the set $J$ to be maximal (for inclusion) among the sets arising from any partition of $R$. (It will also convenient to take $J$ with maximal size rather than merely maximal for inclusion.)

\subsection{The main fragmentation statement}

We now prove our fragmentation statement in the general case. Somewhat similar ideas had been used in Proposition 11.4 from \cite{K. subtensors}, where the $R_2$-rank decomposition was replaced by an assumption that all suitably chosen slices of the tensor have bounded partition rank. The broader idea of splitting tensors into lower-order tensors was also used throughout \cite{K. decompositions} in the much more limited context of slice rank decompositions.

\begin{proposition} \label{Fragmentation} Let $d \ge 2$ be an integer and let $R_1, R_2$ be non-empty families of partitions of $\lbrack d \rbrack$ such that $R_1 \cup R_2 \neq R_{\tr}$. Let $J = J_{\max}(R_1 \cup R_2)$ and assume that $J \neq [d]$. Let $r_1,s_1,r_2,s_2$ be nonnegative integers. Assume that we have an equality \begin{equation} \sum_{i=1}^{r_1} A_i^1(x(J^c)) B_i^1(x(J)) + \sum_{i=1}^{s_1} F_i^1(x) = \sum_{i=1}^{r_2} A_i^2(x(J^c)) B_i^2(x(J)) + \sum_{i=1}^{s_2} F_i^2(x) \label{equality between two R-rank decompositions}\end{equation} between a $R_1$-rank and a $R_2$-rank decomposition, where the functions \[A_1^1, \dots, A_{r_1}^1, A_1^2, \dots, A_{r_2}^2\] are all linearly independent, the functions $A_i^1, A_i^2$ respectively have $R_{1\mathrm{comp}}(J)$-rank at most $1$ and $R_{2\mathrm{comp}}(J)$-rank at most $1$, the functions $B_i^1, B_i^2$ are arbitrary, and the functions $F_i^1, F_i^2$ respectively have $R_{1-}(J)$-rank at most $1$ and $R_{2-}(J)$-rank at most $1$. Then there exists a subset $U \subset \prod_{j \in J^c} [n_j]$ with size $r_1+r_2$ such that the families of tensors \begin{align*}\cf_{11} & = (A_i^1(F_{i'}^1)_u: i \in [r_1], i' \in [s_1], u \in U)\\
\cf_{12} & = (A_i^1(F_{i'}^2)_u: i \in [r_1], i' \in [s_2], u \in U)\\
\cf_1 & = (F_{i}: i \in [s_1])\end{align*} all consist in tensors that have $R_1(J)'$-rank at most $1$ and are such that either side of \eqref{equality between two R-rank decompositions} belongs to the linear span of the union \[\cf_{11} \cup \cf_{12} \cup \cf_1.\] In particular if $r_1 + s_1 \le k_1$ and $r_2 + s_2 \le k_2$ for some nonnegative integers $k_1, k_2$ then the $R_1'(J)$-rank of $T$ is at most \[k_1 ((k_1+k_2)^2 + 1).\]

\begin{proof} We define a family \[(A_1^{1*}, \dots, A_{r_1}^{1*},A_1^{2*}, \dots, A_{r_2}^{2*})\] of dual functions to the family \[(A_1^1, \dots, A_{r_1}^1, A_2^1, \dots, A_{r_2}^2).\] We can require the dual family to be supported inside some subset $U \subset \prod_{j \in J^c} [n_j]$ with size $r_1+r_2$. For each $i \in \lbrack r_1 \rbrack$, applying $A_i^{1*}$ to \eqref{equality between two R-rank decompositions} we obtain  \begin{align*} B_i^1(x(J)) & = \sum_{i'=1}^{s_2} (A_i^{1*}.F_{i'}^2)(x(J)) - \sum_{i'=1}^{s_1} (A_i^{1*}.F_{i'}^1)(x(J)) \\
& = \sum_{i'=1}^{s_2} \sum_{u \in U} A_i^{1*}(u) F_{i'}^2(u, x(J)) - \sum_{i'=1}^{s_1} \sum_{u \in U} A_i^{1*}(u) F_{i'}^1(u, x(J)). \end{align*} Let $i' \in [s_1]$ and $u \in U$ be fixed. Because $J$ is maximal (for the ordering $\le$, so in particular for inclusion) among all parts in all partitions of $R_1 \cup R_2$, and the function $F_{i'}^1$ has $R_{1-}(J)$-rank equal to $1$, it can be written as a product \[ \prod_{J' \in P} a_{J'}(x(J')) \] for some partition $P$ of $R_1$ such that no $J' \in P$ contains $J$. So $J$ has non-empty intersection with at least two distinct parts $J_1,J_2$ of $P$. Substituting $u_j$ for $x_j$ for every $j \in J$ shows that the slice $(F_{i'}^1)_u$ hence has partition rank at most $1$. Likewise for the slices $(F_{i'}^2)_u$. Because the functions $A_i^1$ have $R_{1\mathrm{comp}}(J)$-rank at most $1$, we conclude that the products \[A_i^1 (F_{i'}^1)_u, A_i^1 (F_{i'}^2)_u\] all have $R_{1{\mathrm{new}}}(J)$-rank at most $1$ (and hence $R_1'(J)$-rank at most $1$). The functions $F_{i'}^1$ all have $R_{1-}(J)$-rank at most $1$, so also have $R_1'(J)$-rank at most $1$. So the $R_1'(J)$-rank of $T$ is at most \[r_1s_1(r_1+r_2) + r_1s_2(r_1+r_2) + s_1\] and the desired bound follows. \end{proof}

\end{proposition}

\section{Power lower bounds on ranks of products of identity tensors} \label{Section: Lower bounds on ranks of products of diagonal tensors}

In this section we prove Theorem \ref{Diagonal bound theorem}.

\subsection{Some more examples}

Before writing out the proof of Theorem \ref{Diagonal bound theorem} in the general case, let us mention a few situations where we may conclude in rather short order.

\begin{example} \label{Example 1} \emph{As discussed in the introduction, whenever $d \ge 2$ and $R$ is a non-empty family of partitions of $[d]$ not containing the trivial partition $\{[d]\}$, the order-$d$ “identity" tensor $\Delta_d$ has $R$-rank greater than or equal to its partition rank, which is equal to $n$.}\end{example}

\begin{example} \label{Example 2} \emph{If $R$ contains only one partition $Q$, and $P$ is not finer than $Q$, then there exists some $J_0 \in P$ which has non-empty intersection with two distinct parts $J_1, J_2 \in Q$. Assume that we have some $\{Q\}$-rank decomposition \begin{equation} \sum_{i=1}^k \prod_{J \in Q} a_{J,i}(x(J)) = \Delta_P. \label{decomposition in example 2} \end{equation} for some nonnegative integer $k$. We then choose an element $y \in \prod_{j \in J_0^c} [n_j]$ such that $y_{j_1} = j_{j_2}$ whenever $j_1,j_2 \in J_0^c$ are in the same part of $P$ (for instance, we can even choose $y$ such that $y_j$ is the same for all $j \in J_0^c$). Specialising the equality \eqref{decomposition in example 2} to the order-$|J_0|$ slices obtained by fixing $x_j$ to be equal to $y_j$ for every $j \in J_0^c$ we obtain \[\sum_{i=1}^k \prod_{J \in Q} a_{J,i}(x(J \cap J_0), y(J \setminus J_0)) = \Delta_{J_0}(x(J_0)).\] Because $J_1 \cap J_0$ and $J_2 \cap J_0$ are not empty, every summand on the left-hand side has partition rank at most $1$, and the left-hand side hence has partition rank at most $k$. Meanwhile we know from Example \ref{Example 1} that the right-hand side has partition rank $n$. So $k \ge n$ and the tensor $\Delta_P$ therefore has $\{Q\}$-rank (and equivalently $R$-rank) at least $n$.} \end{example}

\begin{example} \label{Example 3} \emph{We now allow $R$ to contain more than one partition, but require that there exists some $J_0 \in P$ that is never contained in $J$ for any $Q \in R$ and every $J \in Q$. Assume that we have some $R$-rank decomposition \begin{equation} \sum_{Q \in R} \sum_{i=1}^{r_Q} \prod_{J \in Q} a_{Q,J,i}(x(J)) = \Delta_P \label{decomposition in example 3} \end{equation} for some nonnegative integers $r_Q$ for every $Q \in R$. Then choosing $y$ as in Example \ref{Example 2} and taking the same slice as there we obtain \[\sum_{Q \in R} \sum_{i=1}^{r_Q} \prod_{J \in Q} a_{Q,J,i}(x(J \cap J_0), y(J \setminus J_0)) = \Delta_{J_0}(x(J_0)).\] Our assumption on $J_0$ again ensures that every summand on the left-hand side has partition rank at most $1$ and hence that the left-hand side has partition rank at most $\sum_{Q \in R} r_Q$. Since the right-hand side has partition rank equal to $n$, we have shown that $\Delta_P$ has $R$-rank at least $n$.} \end{example}

\begin{example} \label{Example 4} \emph{At the other extreme, let us assume that there exists a common part $J_0 \subset [d]$ which belongs to $P$ and to every partition $Q \in R$. With a view towards what will be our eventual inductive argument on $d$ in the general case, also assume that Theorem \ref{Diagonal bound theorem} has been shown for all $d'<d$. Suppose that we have some $R$-rank decomposition \eqref{decomposition in example 3} of $\Delta_P$. We may then rewrite it as \begin{equation} \sum_{Q \in R} \sum_{i=1}^{r_Q} A_{Q,i}(x(J_0^c)) a_{Q,J_0,i}(x(J_0)) = \Delta_P \label{decomposition in example 4} \end{equation} for some order-$(|J^c|)$ tensors $A_{Q,i}: \prod_{j \in J_0^c} [n_j] \to \F$. Let $\Delta_{P \setminus \{J_0\}}: \prod_{j \in J_0^c} [n_j] \to \F$ be the tensor defined by \[\Delta_{P \setminus \{J_0\}}(x(J_0^c)) = \prod_{J \in P \setminus \{J_0\}} \Delta_J(x(J)).\] Choosing an element $y \in \prod_{j \in J_0} [n_j]$ such that $y_{j_1} = y_{j_2}$ whenever $j_1,j_2 \in J_0$ and specialising the equality \eqref{decomposition in example 4} to the order-$|J_0^c|$ slices obtained by fixing $x_j$ to be equal to $y_j$ for every $j \in J_0$ we obtain \[ \sum_{Q \in R} \sum_{i=1}^{r_Q} \lambda_{Q,i} A_{Q,i}(x(J_0^c)) = \Delta_{P \setminus \{J_0\}} \] for some $\lambda_{Q,i} \in \F$. We write $R_{\mathrm{comp}}(J_0)$ for the family of partitions $\{Q \setminus \{J_0\}: Q \in R\}$. We have shown that for every nonnegative integer $k$, if $\Delta_P$ has $R$-rank at most $k$ then $\Delta_{P \setminus \{J_0\}}$ has $R_{\mathrm{comp}}(J_0)$-rank at most $k$. Because $P$ is not finer than any of the partitions $Q \in R$, and the part $J_0$ belongs to $P$ and to all $Q \in R$, it follows that $P \setminus \{J_0\}$ is not not finer than any of the partitions in $R_{\mathrm{comp}}(J_0)$. By Theorem \ref{Diagonal bound theorem} applied to $|J_0^c|$, $P \setminus \{J_0\}$, and $R_{\mathrm{comp}}(J_0)$ and conclude that $\Rrk \Delta_{P} \ge n^{c(|J_0^c|)}$.}

\end{example}

\subsection{Proof of the power lower bound in the general case}

Having been largely inspired by the previous examples we now prove Theorem \ref{Diagonal bound theorem} in full generality.

\begin{proof}[Proof of Theorem \ref{Diagonal bound theorem}]

If $P$ is the trivial partition $\{[d]\}$, then $\Delta_P$ is the order-$d$ “identity” tensor $\Delta_d$, and the result follows from Example \ref{Example 1}. We now assume that $P \neq \{[d]\}$. Because $P$ is not finer than any partition in $R$, the set $R$ does not contain the partition $\{[d]\}$. We proceed by induction on the pair $(d,J_{\max}(R))$. 

The base case is the case $d=2$. We now describe the inductive step. Assume that for some value of $d \ge 3$, the result is proved for any pair $(d^*,R^*)$ (where $R^*$ is a non-empty family of partitions of $[d^*]$) with $2 \le d^* < d$ or with $d^*=d$ and $J_{\max}(R^*) < J_{\max}(R)$. We then distinguish two cases.

\textbf{Case 1:} There exists a part $J \in P$ which is not strictly contained in any part of any partition in $R$. We can write \[\Delta_P(x) = \Delta_{P \setminus \{J\}}(x(J^c))  \Delta_J(x(J)).\] Let $k$ be a nonnegative integer. Assume that the $R$-rank of $\Delta_P$ is at most $k$. Then we may write an $R$-rank decomposition \begin{equation} \sum_{i=1}^r A_i(x(J^c)) B_i(x(J)) + \sum_{i=1}^{s} F_i(x) = \Delta_{P \setminus \{J\}}(x(J^c)) \Delta_J(x(J))  \label{R-rank decomposition of Delta_P} \end{equation} of $\Delta_P$ where $r+s \le k$, the functions $A_i$ have $R_{\mathrm{comp}}(J)$-rank at most $1$, the functions $B_i$ are arbitrary, and the functions $F_i$ have $R_-(J)$-rank at most $1$.  Because the partition $P$ is not finer than any partition of $R$, the partition $P \setminus \{J\}$  is not finer than any partition of $R_{\mathrm{comp}}(J)$. By the inductive hypothesis the $R_{\mathrm{comp}}(J)$-rank of $\Delta_{P \setminus \{J\}}$ is at least $n^{c(|J^c|)}$, so the subadditivity of the $R_{\mathrm{comp}}(J)$-rank shows that whenever \[r \le k < n^{c(|J^c|)}\] the tensor $\Delta_{P \setminus \{J\}}$ is necessarily outside the linear span of $A_1,\dots,A_r$ and there hence exists a function $u: [n]^{J^c} \to \F$ with support $U$ with size at most $r+1 \le k+1$ such that $u.A_i = 0$ for each $i \in [r]$ but $u.\Delta_{P \setminus \{J\}} = 1$. Applying $u$ to both sides of \eqref{R-rank decomposition of Delta_P} provides \begin{equation} \sum_{i=1}^{s} (u.F_i)(x(J)) = \Delta_J(x(J)).\label{resulting equality with Delta_J} \end{equation} We may compute the left-hand side of \eqref{resulting equality with Delta_J} as \[\sum_{i=1}^{s} \sum_{y \in U} u(y) (F_{i})(x(J),y).\] It follows from the definition of $R_-(J)$ that every slice $x(J) \to F_{i}(x(J),y)$ has partition rank at most $1$, so the left-hand side of \eqref{resulting equality with Delta_J} has partition rank at most $s(r+1) \le k^2$. However, the partition rank of $\Delta_J$ is equal to $n$, so $k^2 \ge n$. Therefore, in Case 1 we conclude \[\Rrk \Delta_P \ge \min(n^{c(|J^c|)}, n^{1/2}).\]

\textbf{Case 2:} We are not in Case 1, in other words every part of $P$ is strictly contained in some part of some partition of $R$. We apply Proposition \ref{Fragmentation} by taking $R_1 = R$, $R_2 = \{P\}$, and $J = J_{\max}(R)$. We next show that its assumptions are satisfied.

Our Case 2 assumption applied to an arbitrary part of $P$ shows that that part is strictly contained in some part of some partition in $R$; that latter part hence has size at least 2, so $R \neq R_{\tr}$ and hence $R_1 \cup \{P\} \neq R_{\tr}$.

Our Case 2 assumption also shows that in particular if $J_0$ is a part of $P$ with largest size, then it is strictly contained in some part of some partition of $R$, so \[\max \{|J’|: J’ \in Q, Q \in R\} > \max \{|J’|: J’ \in P\}\] and hence \[J_{\max}(R \cup \{P\}) = J_{\max}(R).\] 

Again, by our Case 2 assumption the part $J = J_{\max}(R)$ must have non-empty intersection with at least two parts of the partition $P$: indeed if this were not the case, then $J$ would be contained in some part $J_0 \in P$; since by assumption the part $J_0$ is strictly contained in some part $J_1$ of some partition of $R$, we would then have $|J| \le |J_0| < |J_1|$, contradicting that $J$ has maximal size among the parts of all partitions of $R$ and hence contradicting that $J=J_{\max}(R)$. In other words we have $(\{P\})_-(J) = \{P\}$ and hence $\Delta_P$ has $(\{P\})_-(J)$-rank at most $1$.

From Proposition \ref{Fragmentation} we obtain that if $\Rrk \Delta_P \le k$ for some nonnegative integer $k$, then the $R'(J)$-rank of $\Delta_P$ is at most \[k ((k+1)^2+1) \le (k+1)^3.\] This finishes Case 2.

We iterate Case 2 until we can no longer do so. The number of times that Case 2 may be iterated is bounded by the number of subsets of $[d]$, which is at most $2^d$. At the end of the argument we hence obtain the bound $\Rrk \Delta_P \ge n^{c(d)}$ with $c(d) = 3^{-2^{d+1}}$. \end{proof}

\section{Boundedness of ranks generated by a pair of ranks} \label{Boundedness of ranks generated by a pair of ranks}

This section is devoted to the proof of Theorem \ref{Boundedness for two ranks}.

\subsection{The case of ranks arising from one partition}

We begin by proving Theorem \ref{Boundedness for two ranks} in the special case where $R_1, R_2$ each only contain one partition.

\begin{proposition}

Let $d \ge 2$ be an integer, let $P,Q$ be two partitions of $[d]$, and let $r,s$ be nonnegative integers. Assume that $T$ is an order-$d$ tensor with $\{P\}$-rank and $\{Q\}$-rank equal to respectively $r$ and $s$. Then $T$ has $\{P\} \wedge \{Q\}$-rank at most $(rs)^{|P|} \le (rs)^d$.

\end{proposition}

\begin{proof} We consider two respective $\{P\}$-rank and $\{Q\}$-rank decompositions \begin{equation} \sum_{i=1}^r \prod_{J_1 \in P} a_{i,J_1}(x(J_1)) \text{ and }  \sum_{j=1}^s \prod_{J_2 \in Q} b_{j,J_2}(x(J_2)) \label{two decompositions each with one partition} \end{equation} of $T$, both with minimal lengths. Let $J_1 \in P$ be fixed. Then the family of products \[A_{1,J_1} = \prod_{J \in P \setminus \{J_1\}} a_{1,J}, \dots, A_{r,J_1} = \prod_{J \in P \setminus \{J_1\}} a_{r,J} \] with $i \in [r]$ is linearly independent, as can be seen by assuming for contradiction that this is not the case and rewriting the $\{P\}$-rank decomposition from \eqref{two decompositions each with one partition} into a decomposition with smaller length. We fix a family $(A_{1,J_1}^*, \dots, A_{r,J_1}^*)$ of dual functions to the family $(A_{1,J_1}, \dots, A_{r,J_1})$ supported inside some subset $U_{J_1} \subset \prod_{j \in J_1^c} [n_j]$ with size $r$. For every $i \in [r]$, applying $A_{i,J_1}^*$ to the equality between the two decompositions \eqref{two decompositions each with one partition} we get \[a_{i,J_1} = \sum_{j=1}^s (A_{i,J_1}^*. \prod_{J_2 \in Q} b_{j,J_2})\] and hence \[a_{i,J_1}(x(J_1)) =  \sum_{j=1}^s \sum_{u_{J_1} \in U_{J_1}} A_{i,J_1}^*(u_{J_1}) \prod_{J_2 \in Q} b_{j,J_2}(x(J_1 \cap J_2), u_{J_1}(J_2 \setminus J_1)).\] Substituting in the first of the decompositions \eqref{two decompositions each with one partition} we obtain \[T(x_1, \dots, x_d) = \sum_{i=1}^r \prod_{J_1 \in P} \left( \sum_{j=1}^s \sum_{u_{J_1} \in U_{J_1}} A_{i,J_1}^*(u_{J_1}) \prod_{J_2 \in Q} b_{j,J_2}(x(J_1 \cap J_2), u_{J_1}(J_2 \setminus J_1)) \right). \] This decomposition can in turn be expanded as \[ \sum_{j \in [s]^{P}} \sum_{(u_{J_1})_{J_1 \in P}}\left(\sum_{i=1}^r \prod_{J_1 \in P} A_{i,J_1}^*(u_{J_1})\right) \left( \prod_{J_1 \in P} \prod_{J_2 \in Q}  b_{j,J_2}(x(J_1 \cap J_2), u_{J_1}(J_2 \setminus J_1)) \right).\] The family \[\left( \prod_{J_1 \in P} \prod_{J_2 \in Q}  b_{j,J_2}(x(J_1 \cap J_2), u_{J_1}(J_2 \setminus J_1)): j \in [s]^{P}, (u_{J_1})_{J_1 \in P} \in \prod_{J_1 \in P} U_{J_1} \right) \] is a family of at most $r^{|P|} s^{|P|}$ (order-$d$) tensors with $\{P_1 \wedge P_2\}$-rank at most $1$. The claim follows. \end{proof}

\subsection{Proof of boundedness in the general case}

\begin{proof}[Proof of Theorem \ref{Boundedness for two ranks}]

We finally prove Theorem \ref{Boundedness for two ranks} in the general case of arbitrary $R_1, R_2$.

Let $R_1$, $R_2$ be non-empty families of partitions of $[d]$. If $R_1$ or $R_2$ contains the trivial partition $\{[d]\}$, then we can assume without loss of generality that $R_1$ does. The set $R_1 \wedge R_2$ then contains $R_2$, so the $R_1 \wedge R_2$-rank is at most the $R_2$-rank (in fact, both are equal) and the result is immediate. Throughout the remainder of the proof we assume that neither $R_1$ nor $R_2$ contains the partition $\{[d]\}$. This proof proceeds by induction on the pair $(d,J_{\max}(R_1 \cup R_2))$. 

The base case is the case where $R_1$ and $R_2$ are both equal to $R_{\tr}$. Then in particular $R_1 = R_2$, so $R_1 \wedge R_2 = R_1 = R_2$, and a $R_1$-rank decomposition of $T$ with length $k_1$ hence also provides a $R_1 \wedge R_2$-rank decomposition of $T$ with length $k_1$.

We now describe the inductive step. Assume that for some value of $d \ge 3$ and for some non-empty families $R_1$, $R_2$ of partitions of $[d]$ with $R_1 \cup R_2 \neq R_{\tr}$, the result is proved for any triple $(d^*,R_1^*, R_2^*)$ (where $R_1^*, R_2^*$ are non-empty families of partitions of $[d^*]$) satisfying $2 \le d^* < d$ or satisfying $d^*=d$ and $J_{\max}(R_1^* \cup R_2^*) < J_{\max}(R_1 \cup R_2)$.

Let $T$ be an order-$d$ tensor with two respective $R_1$-rank and $R_2$-rank decompositions \begin{align} \sum_{i=1}^{r_1} A_i^1(x(J^c)) B_i^1(x(J)) + \sum_{i=1}^{s_1} F_i^1(x) \label{decomposition 1} \\ \sum_{i=1}^{r_2} A_i^2(x(J^c)) B_i^2(x(J)) + \sum_{i=1}^{s_2} F_i^2(x) \label{decomposition 2} \end{align} with respective lengths $k_1 = r_1+s_1$ and $k_2 = r_2+s_2$, where as usual the functions $A_i^1, A_i^2$ have $R_{1\mathrm{comp}}(J)$-rank and $R_{2\mathrm{comp}}(J)$-rank at most $1$ respectively, and the functions $F_i^1, F_i^2$ have $R_{1-}(J)$-rank and $R_{2-}(J)$-rank at most $1$ respectively. Let $J = J_{\max}(R_1 \cup R_2)$. As ever we may assume that the families of functions $(A_1^1,\dots, A_{r_1}^1)$ and $(A_1^2,\dots, A_{r_2}^2)$ are each linearly independent in \eqref{decomposition 1} and \eqref{decomposition 2}.

If the functions \begin{equation} A_1^1, \dots, A_{r_1}^1, A_1^2, \dots, A_{r_2}^2 \label{special case where the functions are all linearly independent} \end{equation} are all linearly independent then we conclude the inductive step in rather short order. Indeed Proposition \ref{Fragmentation} applied to the decompositions \eqref{decomposition 1} and \eqref{decomposition 2} then shows that the $R_1’(J)$-rank of $T$ is at most $k_1 ((k_1+k_2)^2 + 1)$ and likewise that the $R_2’(J)$-rank of $T$ is at most $k_2 ((k_1+k_2)^2 + 1)$. Since \[J_{\max}(R_1’(J), R_2’(J)) < J\] the inductive hypothesis then shows that the $R_1’(J) \wedge R_2’(J)$-rank of $T$ is at most \[C_{d, R_1’(J), R_2’(J)}(k_1 ((k_1+k_2)^2 + 1), k_2 ((k_1+k_2)^2 + 1)).\] Since $R_1’(J) \prec R_1$ and $R_2’(J) \prec R_2$ we have \[R_1’(J) \wedge R_2’(J) \prec R_1 \wedge R_2,\] so by Proposition \ref{Refinements do not decrease rank} the same bound on the $R_1 \wedge R_2$-rank of $T$ then follows.

The general case of the inductive step, where the functions \eqref{special case where the functions are all linearly independent} are not assumed to be linearly independent, proceeds by reducing to this case, through another, inner induction, as we will now explain. This inner induction will involve working with the first parts \begin{align} \sum_{i=1}^{r_1} & A_i^1(x(J^c)) B_i^1(x(J)) \label{first part of decomposition 1} \\ \sum_{i=1}^{r_2} & A_i^2(x(J^c)) B_i^2(x(J)) \label{first part of decomposition 2} \end{align} of the decompositions \eqref{decomposition 1} and \eqref{decomposition 2} respectively, and will not involve the functions $F_i^1, F_i^2$.

Assume that the functions \eqref{special case where the functions are all linearly independent} are not all linearly independent. Then we have \begin{equation} \sum_{i=1}^{r_1} \a_i A_i^1 = \sum_{i=1}^{r_2} \b_i A_i^2 \label{common tensor in the first linear dependence} \end{equation} for some $\a \neq 0$ and $\b \neq 0$. Applying Theorem \ref{Boundedness for two ranks} with $|J^c|$, $R_{1 \mathrm{comp}}(J)$ and $R_{2 \mathrm{comp}}(J)$ we obtain without loss of generality that both sides of \eqref{common tensor in the first linear dependence} can be written as a linear combination of tensors in some family $\ct_{\wedge, 1}(J)$ of at most \[t_1 \le C_{|J^c|, R_{1 \mathrm{comp}}(J), R_{2 \mathrm{comp}}(J)}(r_1, r_2)\] tensors each with $R_{1 \mathrm{comp}}(J) \wedge R_{2 \mathrm{comp}}(J)$-rank equal to $1$. We write \[\ct_{\wedge, 1}(J) = (A_1, \dots, A_{t_1}).\] Without loss of generality we can assume that all functions in $\ct_{\wedge, 1}(J)$ are linearly independent and that $\a_{r_1}^1 \neq 0$. The equality \eqref{common tensor in the first linear dependence} hence allows us to rewrite $A_{r_1}^1$ as a linear combination of $A_{1}^1, \dots, A_{r_1-1}^1$ and of the tensors in $\ct_{\wedge, 1}(J)$, so we may rewrite \eqref{first part of decomposition 1} without involving the function $A_{r_1}$, as \[ \sum_{i=1}^{r_1-1} A_i^1(x(J^c)) B_i^1(x(J)) + \sum_{i=1}^{t_1} \mu_i A_i(x(J^c)) B_{r_1}^1(x(J)) \] for some $\mu_i \in \F$. If the functions \begin{equation} A_1^1, \dots, A_{r_1-1}^1, A_2^1, \dots, A_{r_2}^2, A_1, \dots, A_{t_1} \label{linear independence after first linear dependence} \end{equation} are all linearly independent, then we stop.

If on the other hand the functions \eqref{linear independence after first linear dependence} are linearly dependent, then we move to the next step in the inner induction as follows. We have a linear relation of the type \begin{equation} \sum_{i=1}^{r_1-1} \a_i A_i^1 = \sum_{i=1}^{r_2} \b_i A_i^2 + \sum_{i=1}^{t_1} \g_i A_i \label{common tensor in the second linear dependence} \end{equation} for some $\a, \b, \g$ not all zero. Applying a second time Theorem \ref{Boundedness for two ranks} with $|J^c|$, $R_{1 \mathrm{comp}}(J)$ and $R_{2 \mathrm{comp}}(J)$ to \eqref{common tensor in the second linear dependence} shows that both sides are linear combinations of tensors in some family $\ct_{\wedge, 2}(J)$ of at most \[t_2 \le C_{|J^c|, R_{1 \mathrm{comp}}(J), R_{2 \mathrm{comp}}(J)}(r_1, r_2 + t_1)\] tensors each with $R_{1 \mathrm{comp}}(J) \wedge R_{2 \mathrm{comp}}(J)$-rank equal to $1$. Without loss of generality we can assume that the tensors in $\ct_{\wedge, 2}(J)$ are linearly independent, and that $\ct_{\wedge, 2}(J)$ contains $A_1, \dots, A_{t_1}$, at the cost of increasing the bound on the size of $\ct_{\wedge, 2}(J)$ by $t_1$. We write \[\ct_{\wedge, 2}(J) = (A_1, \dots, A_{t_2}).\] Because the functions $A_1, \dots, A_{t_1}$ are linearly independent we have $\a \neq 0$ or $\b \neq 0$. If $\a \neq 0$ then without loss of generality we can write $A_{r_1-1}^1$ as a linear combination of \begin{equation} A_1^1, \dots, A_{r_1-2}^1, A_1^2, \dots, A_{r_2}^2, A_1, \dots, A_{t_2}. \label{linear independence after second step, situation 1} \end{equation} Otherwise without loss of generality we can write $A_{r_2}^2$ as a linear combination of \begin{equation} A_1^1, \dots, A_{r_1-1}^1, A_1^2, \dots, A_{r_2-1}^2, A_1, \dots, A_{t_2} \label{linear independence after second step, situation 2}.\end{equation} These situations respectively lead to \eqref{first part of decomposition 1} (resp. \eqref{first part of decomposition 2}) being rewritten respectively as \begin{align*} \sum_{i=1}^{r_1-2} A_i^1(x(J^c)) B_i^1(x(J)) &+ \sum_{i=1}^{t_2} \mu_i' A_i(x(J^c)) B_{r_1-1}^1(x(J)) + \sum_{i=1}^{t_1} \mu_i A_i(x(J^c)) B_{r_1}^1(x(J)) \\ \sum_{i=1}^{r_2-1} A_i^2(x(J^c)) B_i^2(x(J)) &+ \sum_{i=1}^{t_2} \mu_i'' A_i(x(J^c)) B_{r_2}^2(x(J)) \end{align*} for some $\mu_i'\in \F$ (resp. for some $\mu_i'' \in \F$). In the first and second of these two situations we then stop if the functions \eqref{linear independence after second step, situation 1}, resp. \eqref{linear independence after second step, situation 2} are linearly independent, and otherwise we continue. As long as we continue, after $\tau$ iterations in the inner induction we have without loss of generality rewritten \eqref{first part of decomposition 1} and \eqref{first part of decomposition 2} respectively as \begin{align} \sum_{i=1}^{r_1-m_1} A_i^1(x(J^c)) B_i^1(x(J)) & + \sum_{i=1}^{t_\tau} A_{i}(x(J^c)) B_i^{1, \tau}(x(J)) \label{decomposition 1 after general step} \\
\sum_{i=1}^{r_2-m_2} A_i^2(x(J^c)) B_i^2(x(J)) & + \sum_{i=1}^{t_\tau} A_{i}(x(J^c)) B_i^{2, \tau}(x(J)) \label{decomposition 2 after general step} \end{align} for some linearly independent family \[\ct_{\wedge, \tau}(J) = (A_1, \dots, A_{t_\tau}),\] some nonnegative integers $m_1, m_2$ satisfying $m_1 + m_2 = \tau$, and some functions $B_i^{1, \tau}, B_i^{2, \tau}$. If the functions \begin{equation} A_1^1, \dots, A_{r_1-m_1}^1, A_2^1, \dots, A_{r_2-m_2}^2, A_1, \dots, A_{t_\tau} \label{functions linearly independent at the end} \end{equation} are linearly independent then we stop. 

Otherwise we write a linear dependence relation \begin{equation} \sum_{i=1}^{r_1-m_1} \a_i A_i^1 = \sum_{i=1}^{r_2-m_2} \b_i A_i^2 + \sum_{i=1}^{t_{\tau}} \g_i A_i \label{common tensor in the linear dependence of the general step} \end{equation} for some $\a, \b, \g$, with $\a \neq 0$ or $\b \neq 0$. Again, applying Theorem \ref{Boundedness for two ranks} with $|J^c|$, $R_{1 \mathrm{comp}}(J)$ and $R_{2 \mathrm{comp}}(J)$ to \eqref{common tensor in the linear dependence of the general step} shows that both sides are linear combinations of the tensors in some extension \[\ct_{\wedge, \tau+1}(J) = (A_1, \dots, A_{t_\tau})\] of $\ct_{\wedge, \tau}(J)$ with \[t_\tau \le C_{|J^c|, R_{1 \mathrm{comp}}(J), R_{2 \mathrm{comp}}(J)}(r_1, r_2 + t_{\tau}).\] If $\a \neq 0$ or $\b \neq 0$ then we may rewrite \eqref{first part of decomposition 1} (resp. \eqref{first part of decomposition 2}) as \eqref{decomposition 1 after general step} (resp. \eqref{decomposition 2 after general step}) with $m_1$ (resp. $m_2$) incremented by $1$, at the cost of the range of the second sum extending from $[t_{\tau}]$ to $[t_{\tau + 1}]$ (and with new functions $B_i^{1, \tau+1}$, $B_i^{2, \tau+1}$). This finishes the iteration $\tau + 1$ of the inner induction.

If we have not stopped before the end of the iteration $r_1+r_2$, then we necessarily stop thereafter, since for every $\tau$ the family $(A_1, \dots, A_{t_{\tau}})$ is linearly independent and hence so is $(A_1, \dots, A_{t_{r_1+r_2}})$ in particular. Assume that we have stopped the inner induction after $\tau$ iterations, and have obtained \eqref{decomposition 1 after general step} and \eqref{decomposition 2 after general step}. We write $t = t_{\tau}$ and now describe the remainder of the outer inductive step. 

We have established that the tensor \[T' = T -\sum_{i=1}^t A_{i}(x(J^c)) B_{i}^{2, \tau} (x(J)) \] admits the $R_1$-rank and $R_2$-rank decompositions \begin{align*} \sum_{i=1}^{r_1-m_1} A_i^1(x(J^c)) B_i^1(x(J)) &+ \sum_{i=1}^{s_1} F_i^1(x) + \sum_{i=1}^t A_{i}(x(J^c)) B_i(x(J))\\ \sum_{i=1}^{r_2-m_2} A_i^2(x(J^c)) B_i^2(x(J)) &+ \sum_{i=1}^{s_2} F_i^2(x) \end{align*} for some functions $B_i$. The functions \eqref{functions linearly independent at the end} are linearly independent, so applying Proposition \ref{Fragmentation} to these two decompositions shows that $T’$ has $R_1’(J)$-rank at most \[(k_1 + t)((k_1+t+k_2)^2 + 1),\] and similarly that $T’$ has $R_2’(J)$-rank at most \[k_2((k_1+t+k_2)^2 + 1).\] The inductive hypothesis then shows that the $R_1’(J) \wedge R_2’(J)$-rank of $T’$ is at most \begin{equation} C_{d, R_1’(J), R_2’(J)}((k_1 + t)((k_1+t+k_2)^2 + 1), k_2 ((k_1+t+k_2)^2 + 1)) \label{almost last bound} \end{equation} and the $R_1 \wedge R_2$-rank of $T’$ is hence bounded above by the same value. Since the functions $A_i$ with $i \in [t]$ have $R_{1\mathrm{comp}}(J) \wedge R_{2\mathrm{comp}}(J)$-rank equal to $1$, the tensors $A_i B_i$ with $i \in [t]$ each have $R_1 \wedge R_2$-rank equal to $1$, so the difference $T-T'$ has $R_1 \wedge R_2$-rank at most $t$. By subadditivity of the rank we then conclude that the $R_1 \wedge R_2$-rank of $T$ is at most $t$ greater than \eqref{almost last bound}.

This completes the outer inductive step. As there are at most $2^d$ possibilities for $J_{\max}(R_1 \cup R_2)$, the number of iterations of that outer inductive step is at most that, within a given value of $d$. \end{proof}

\end{document}